\documentclass[11pt]{amsart}

\usepackage{amsmath, amssymb, mathtools, mathpazo, euler, microtype}
\usepackage{graphicx}
\usepackage[usenames]{xcolor}
\usepackage{float}
\usepackage[inline]{asymptote}

\usepackage[lined,algonl,boxed,norelsize]{algorithm2e}
\usepackage[noend]{algpseudocode} 

\usepackage[margin=1 in]{geometry}
\definecolor{darkblue}{rgb}{0, 0, 0.7}
\definecolor{darkred}{rgb}{0.7, 0, 0}
\usepackage[colorlinks,linkcolor=darkred, citecolor=darkblue, pagebackref=true,pdftex]{hyperref}

\newcommand\defi[1]{\textit{\color{blue}#1}}    

\newcommand\R{\mathbb{R}}
\newcommand\Z{\mathbb{Z}}
\newcommand\N{\mathbb{N}}
\newcommand\Q{\mathbb{Q}}

\newcommand\Lap{\operatorname{\mathcal{L}}}
\newcommand\Basis{\operatorname{\mathcal{B}}}
\newcommand\Cut{\operatorname{\mathcal{C}}}
\newcommand\Flow{\operatorname{\mathcal{F}}}

\renewcommand \deg{\operatorname{deg}}
\newcommand\im{\operatorname{im}}
\newcommand\spn{\operatorname{span}}
\newcommand\define{\vcentcolon=}

\newtheorem*{thm*}{Theorem}
\newtheorem{thm}{Theorem}[section]
\newtheorem*{prop*}{Proposition}
\newtheorem{prop}[thm]{Proposition}
\newtheorem{lemma}[thm]{Lemma}

\newtheorem{cor}[thm]{Corollary}
\newtheorem*{cor*}{Corollary}
\newtheorem{defn}[thm]{Definition}

\newtheorem{example}[thm]{Example}
\newtheorem{question}[thm]{Question}
\newtheorem*{question*}{Question}

\begin{asydef}
    unitsize(1cm);
    pair[] points = {(0, 0), (0, 1), (1, 1), (1, 0)};
    int[][] conns = {{0, 1}, {1, 2}, {2, 3}, {3, 0}, {0, 2}};
    int[][] connsT = {{0, 1}, {0, 2}, {0, 3}, {1, 2}, {2, 3}};
    pair[] dirs = {SW, NW, NE, SE};
    pair[] es = {(1/4, 3/4), (3/4, 1/4), (1/2, 5/3)};

    path[] edges;
    for(int i=0; i<conns.length; ++i)
        edges.push(points[conns[i][0]] -- points[conns[i][1]]);

    path[] edgesT;
    for(int i=0; i<connsT.length; ++i)
        edgesT.push(points[connsT[i][0]] -- points[connsT[i][1]]);

    void drawgraph(pen pe=defaultpen) {
        for(int i=0; i<edges.length; ++i)
            draw(edges[i], p=pe);
        for(int i=0; i<points.length; ++i) {
            dot(points[i], p=(i>0?pe:pe+red)+6);
        }
    }

    void drawgrapharrows(pen pe=defaultpen) {
        for(int i=0; i<edgesT.length; ++i)
            draw(edgesT[i], p=pe, MidArrow);
        for(int i=0; i<points.length; ++i) {
            dot(points[i], p=(i>0?pe:pe+red)+6);
        }
    }

    void labelgraph() {
        for(int i=0; i<points.length; ++i) {
            label(points[i], "{\large $v_" + (string)i + "$}", dirs[i]);
        }
    }
    void labelgraph(string[] values) {
        for(int i=0; i<points.length; ++i) {
            label(points[i], "{\large $v_" + (string)i + "," + values[i] + "$}", dirs[i]);
        }
    }
    void drawdualverts(pen pe, real size=6) {
        for(int i=0; i<es.length; ++i)
            dot(es[i], p=pe+size);
    }
    void drawdualedges(pen pe) {
        draw(es[0] -- es[1], p=pe);
        draw(es[0] .. midpoint(edges[0]) .. (-.5, 1.3) ..  es[2], p=pe);
        draw(es[0] .. midpoint(edges[1]) .. es[2], p=pe);
        draw(es[1] .. midpoint(edges[2]) .. (1.5, 1) .. es[2], p=pe);
        draw(es[1] .. midpoint(edges[3]) .. (-.5, -.5) .. (-1, 1) .. (0, 2) .. es[2], p=pe);
    }
    void drawdual(pen pe=defaultpen) {
        drawdualedges(pe);
        drawdualverts(pe);
    }

    pair cis(real angle) {
        return (cos(angle), sin(angle));
    }

    pair a=cis(0.5*pi);
    pair b=cis(0.1*pi);
    pair c=cis(1.7*pi);
    pair d=cis(1.3*pi);
    pair e=cis(0.9*pi);

    void drawfive(real size=3) {
        dot(a, black+size);
        dot(b, black+size);
        dot(c, black+size);
        dot(d, black+size);
        dot(e, black+size);
    }

    pair u=(0, 0);
    pair v=(0, 1);
    pair w=(0, 2);
    pair x=(1, 0);
    pair y=(1, 1);
    pair z=(1, 2);
    void drawsix(real size=3) {
        dot(u, black+size);
        dot(v, black+size);
        dot(w, black+size);
        dot(x, black+size);
        dot(y, black+size);
        dot(z, black+size);
    }
\end{asydef}

\begin{document}

\title{Integral flow and cycle chip-firing on graphs}

\author{Anton Dochtermann, Eli Meyers, Raghav Samavedam, Alex Yi}

\begin{abstract}
Motivated by the notion of chip-firing on the dual graph of a planar graph, we consider `integral flow chip-firing' on an arbitrary graph $G$.  The chip-firing rule is governed by $\Lap^*(G)$, the dual Laplacian of $G$ determined by choosing a basis for the lattice of integral flows on $G$.  We show that any graph admits such a basis so that $\Lap^*(G)$ is an $M$-matrix, leading to a firing rule on these basis elements that is avalanche finite.  This follows from a more general result on bases of integral lattices that may be of independent interest.  Our results provide a notion of $z$-superstable flow configurations that are in bijection with the set of spanning trees of $G$. We show that for planar graphs, as well as for the graphs $K_5$ and $K_{3,3}$, one can find such a flow M-basis that consists of cycles of the underlying graph.  We consider the question for arbitrary graphs and address some open questions.
\end{abstract}

\date{\today}

\maketitle

\section{Introduction}\label{sec:intro}
The classical theory of \emph{chip-firing} involves a simple game played on the vertices of a graph. Versions of these games play a role in physics in the context of self-organized criticality \cite{BakTanWie, Dha}, and also have been studied for their combinatorial properties and underlying algebraic structure \cite{Big, BjoLovSho}. More recently chip-firing has seen connections to various other disciplines including algebraic geometry \cite{BakNor}. We refer to \cite{CorPer} and \cite{Kli} for good introductions to the subject.  

In one variant of the game, a vertex is designated to be the \emph{root}, and a nonnegative configuration of chips ${\bf c} = (c_1, c_2, \dots, c_n) \in{\N}^n$ is placed on the non-root vertices. If the number of chips on a non-root vertex $i$ is greater than its degree, the vertex $i$ can \emph{fire}, in which case the vertex passes chips to each of its neighbors (one for each edge connected it to $i$).

A configuration is \emph{stable} if no non-root vertex can fire, and is \emph{critical} if it is stable and recurrent (see below for proper definitions).  The set of critical configurations on a graph $G$ forms an abelian group $\kappa(G)$ under coordinate-wise addition and \emph{stabilization}, known as the \emph{critical group} of $G$.  One can show that $\kappa(G)$ does not depend on the choice of root vertex in the graph $G$. Although easy to define, $\kappa(G)$ is a subtle invariant of the graph $G$ and for instance there has been recent interest in determining the distribution of critical groups among random graphs \cite{Woo}. It is also well-known that the set of critical configurations on $G$ has a simple duality with the \emph{superstable} configurations of that graph (known in some contexts as \emph{$G$-parking functions}).

The dynamics of chip-firing on a graph $G$ can be encoded by the reduced \emph{Laplacian} matrix $\tilde \Lap = \tilde \Lap(G)$, where the effect of firing a vertex $i$ can be seen to correspond to subtracting the $i$th row of $\tilde \Lap$.  If we think of $\tilde \Lap: \Z^{n-1} \rightarrow \Z^{n-1}$ as a linear operator, one can prove that the critical group $\kappa(G)$ is recovered as the cokernel of the reduced Laplacian:
\[
    \kappa(G) \cong \Z^{n-1}/\im \tilde\Lap.
\]

By the well-known matrix tree theorem, we have that $|\kappa(G)|=|\tau(G)|$, where $\tau(G)$ is the set of spanning trees of $G$ (recall that $G$ is assumed to be connected).  In particular the set of superstable configurations is in bijection with $\tau(G)$. Merino \cite{Mer} has shown that this bijection can be chosen to preserve degree/activity, which connects the theory of chip-firing on a graph $G$ to properties of the Tutte polynomial $T_G(x,y)$.

A more general notion of `chip-firing' can be encoded in any $n \times n$ integer matrix $L$ that has positive entries on the diagonal and non-positive entries on the off-diagonal.  Chip-firing at a site $i$ is then defined by subtracting $L^T {\bf e}_i$, the $i$th row of $L$, from the current configuration.  One is interested in such matrices $L$ with the property that all initial configurations eventually stabilize by repeated firings determined by the matrix $L$.  This `avalanche finite' property  is equivalent to $L$ being an \emph{$M$-matrix} (see Section \ref{sec:Energy}). Building on work of Baker and Shokrieh from \cite{BakSho}, Guzm\'{a}n and Klivans have shown \cite{GuzKli} that a certain `energy minimizing' perspective on $M$-matrices leads to good notions of \emph{$z$-superstable} configurations, where a notion of firing multisets of vertices simultaneously is relevant.  In particular, the set of $z$-superstable configurations are in bijection with the equivalence classes defined by the matrix $L$.

Returning to the world of graphs, we see that if a graph $G$ is embedded in the plane with dual graph $G^*$, then a configuration on the vertices of $G^*$ can naturally be thought be modelled as a `cycle configuration' on $G$.  After fixing an orientation on $G$ and an orientation of the plane, firing a vertex in $G^*$ then corresponds to firing a directed cycle in $G$.  A cycle $C$ sends chips to any cycle $C^\prime$ that it shares at least one edge with, with the total number of chips sent determined by the number of edges that are oriented consistently versus inconsistently.  In addition, we see that the number of `superstable cycle configurations' (i.e. the superstable configurations of $G^*$) are in bijection with the number of spanning trees of $G$, where the number of trees with $k$ \emph{internally} passive edges is equal to the number of superstable cycle configurations of degree $k$.

What happens if $G$ is not planar? In this paper, we study the notion of cycle (and, more generally, integral flow) chip-firing for general (not necessarily planar) graphs.  We adopt the perspective that traditional chip-firing on $G$ can be viewed as choosing a particular basis for the \emph{lattice of integral cuts} $\Cut(G)$, given by all vertex cuts determined by non-root vertices.  Choosing this basis for $\Cut(G)$ then defines the reduced incidence matrix $\tilde \partial$ and $\tilde \partial \circ \tilde \partial^T = \tilde \Lap$ recovers the reduced Laplacian, which we have seen governs the usual rules of chip-firing.  

In the dual case we want to choose a basis for $\Flow(G) = \ker \partial \cap \Z^m$, the integral kernel of the (signed) incidence matrix $\partial$ (sometimes called the \emph{lattice of integral flows} for the graph $G$, see \cite{BacDelNag}).  If ${\mathcal B} = \{{\bf f}_1, \dots, {\bf f}_g\}$ is such a basis we let $\iota$ denote the $m \times g$ matrix with columns given by the ${\bf f}_i$. This  gives rise to a dual Laplacian $ \iota^T \iota = \Lap^*: \Z^g \rightarrow \Z^g$, where $g = |E| - |V| +1$ is the \emph{genus} of the connected graph.

We first observe that this dual Laplacian can be used to recover the critical group of the graph as $\kappa(G) \cong \Z^g/\im \Lap^*$ (see Proposition~\ref{prop:dualgroup}). This result follows from established properties of integral lattices, for example from \cite{BacDelNag}, and is also used in \cite{Jac} (a fact we learned after preparing this note).  In Section~\ref{sec:dual} we give a mostly self-contained proof. 

Although any integral basis for $\Flow(G)$ gives rise to a dual Laplacian $\Lap^*$, this matrix will not in general define a satisfactory avalanche-finite firing rule.   We would like to interpret $\Lap^*$ as governing a chip-firing rule on the chosen basis elements, with associated notions of $z$-superstable configurations and `good' representatives for each chip-firing equivalence class.  For this we need the matrix $\Lap^*$ to be an $M$-matrix, and our next result shows that this can always be achieved.

\newtheorem*{cor:Mmatrix}{Corollary \ref{cor:Mmatrix}}
\begin{cor:Mmatrix}
Suppose $G$ is any graph.  Then there exists a basis ${\mathcal B} = \{{\bf f}_1, \dots, {\bf f}_g\}$ for its lattice of integral flows such that the associated dual Laplacian $\Lap^*$ is an $M$-matrix.
\end{cor:Mmatrix}

A basis for $\Flow(G)$ satisfying the conditions in Corollary~\ref{cor:Mmatrix} will be called an \emph{flow $M$-basis} (or simply \emph{flow $M$-basis}) for $G$. Corollary~\ref{cor:Mmatrix} follows from a more general result (Theorem~\ref{thm:latticebasis}) regarding bases of any integral lattice that may be of independent interest.  Our proof of Theorem~\ref{thm:latticebasis} is constructive, but our algorithm can lead to complicated linear combinations of integer vectors with many large entries. To interpret our basis in terms of `cycle chip-firing', we would like the elements in our flow $M$-basis to encode directed cycles, that is to consist of $0, -1, 1$ entries with the nonzero entries corresponding to the edges of some (simple) cycle in the graph, with the sign determined by their orientation with respect to the fixed orientation on $G$.

For the case of planar graphs we see that such a \emph{cycle $M$-basis} is given by the boundaries of the bounded faces (appropriately oriented) in some embedding of $G$, see Proposition \ref{prop:planarcircuit}. It is then natural to ask if any graph $G$ admits an $M$-basis consisting of only cycles. This question is still open but in Section~\ref{sec:circuit} we show that such bases do exist for the case of $G = K_5$ and $G = K_{3,3}$ (found via a computer search).

\newtheorem*{prop:circuitK5K33}{Proposition \ref{prop:circuitK5K33}}
\begin{prop:circuitK5K33}
The graphs $K_5$ and $K_{3,3}$ both admit cycle $M$-bases.
\end{prop:circuitK5K33}

We also prove some results concerning the structure that cycle $M$-bases must satisfy. For instance we show that if $K_n$ has a cycle $M$-basis then some cycle must contain at least 4 edges, see Proposition~\ref{prop:triangles}. Similarly, any cycle $M$-basis for $K_{m,n}$ with $m \geq n \geq 3$ must contain a $6$-cycle.

A choice of flow $M$-basis for a graph $G$ defines a dual Laplacian $\Lap^*$ that governs an avalanche finite chip-firing process.  By results of \cite{GuzKli} this then leads to notions of $z$-superstable configurations, unique energy-minimizing elements in the equivalence classes defined by $\Lap^*$. By Proposition~\ref{prop:dualgroup} we know the size of this set of equivalence classes and hence we have the following.

\newtheorem*{prop:zsuper}{Proposition \ref{prop:zsuper}}
\begin{prop:zsuper}
Let $G$ be a connected graph, and let ${\mathcal B} = \{{\bf f}_1, \dots, {\bf f}_g\}$ be a flow $M$-basis with associated dual Laplacian $\Lap^*$.  Then the number of $z$-superstable cycle configurations of $G$ (with respect to $\Lap^*$) is given by $|\tau(G)|$, the number of spanning trees of $G$.
\end{prop:zsuper}

In the case of traditional chip-firing the degree sequence of superstable configurations counts the number of spanning trees via the number of \emph{externally passive} edges (and in particular is independent of the choice of sink vertex).  If $G$ is a planar graph then a natural choice of cycle $M$-basis given by a planar embedding leads to a dual Laplacian ${\mathcal L}^*$ whose $z$-superstable configurations count spanning trees by \emph{internal passivity}. In the general case the interpretation is not so clear. In particular the degree sequence of $z$-superstables depends on the choice of integral flow $M$-basis for $\Flow(G)$.  One wonders if there is a notion of `activity' of a spanning tree that depends on the choice of basis for $\Flow(G)$ and that is reflected in the degree sequence of the associated $z$-superstable flow configurations .  We leave this for a future project.

The rest of the paper is organized as follows.  In Section~\ref{sec:intro} we review some basic notions from the classical theory of chip-firing on graphs as well as the theory of generalized chip-firing defined by $M$-matrices.  In Section~\ref{sec:dual} we discuss our approach to integral flow chip-firing in more detail and prove that any graph admits a flow $M$-basis.  In Section~\ref{sec:circuit} we show that $K_5$ and $K_{3,3}$ both admit flow $M$-bases consisting of cycles and discuss some other examples. Here we also show that any cycle $M$-bases for nonplanar $K_n$ and $K_{m,n}$ must have large cycles. We end with some further discussion and open questions.

{\bf Acknowledgements.} The authors would like to thank Johnny Guzm\'an, Carly Klivans, and Chi Ho Yuen for helpful comments and discussions.  We are also grateful to two anonymous referees whose corrections and suggestions lead to substantial improvements in the paper.  Most examples and experiments were computed using SageMath \cite{Sag}, and we thank Dave Perkinson for sharing his code regarding $z$-superstables. The work presented here was initiated as part of Summer Honors Camp within the Mathworks Program at Texas State University.  We thank Mathworks and Texas State for the support and productive working environment.

\section{Basics of classical chip-firing} \label{sec:basics}
Here we fix some notation and recall some results from the classical theory of chip-firing on graphs, for the most part following conventions from \cite{CorPer}.  Throughout the paper we let $G = (V,E)$ denote a connected undirected simple graph on vertex set $V = \{0,1,\dots, n\}$, with the vertex $0$ designated as the root.  For $u \in V$ we let $\deg(u)$ denote the degree of the vertex, so that 
\[
    deg(u) = |\{w \in V: uw \in E\}|.
\]

A \defi{configuration} of chips on the (non-root) vertices of $G$ is a vector $ {\bf c} = (c_1, c_2, \dots, c_n) \in{\N}^n$. A non-root vertex $i$ can \defi{fire} when $c_i \geq \deg(u)$, so that the number of chips is great than or equal to the number of vertices that are adjacent to $i$.  In this case the vertex passes chips to each of its neighbors (one for each edge connecting it to $i$), resulting in a new configuration ${\bf c}^\prime$.   

A first question to ask is which configurations can be reached from a given configuration via these firing rules. A configuration ${\bf c}$  is \defi{stable} if no non-root vertex can fire, i.e. $c_i < \deg(i)$ for all vertices $i = 1,2, \dots, n$.  In this case the root vertex can fire, passing a chip to each vertex adjacent to $0$.  A configuration ${\bf c}$ is \defi{recurrent} if it is stable and reappears in this process; that is after firing the root vertex there exists a sequence of firings that returns to ${\bf c}$.  The set of stable and recurrent configurations form a group $\kappa(G)$ under coordinatewise addition (and stabilization), called the \defi{critical group} of $G$.  It is well known that $\kappa(G)$ is a finite abelian group whose order is given by the number of spanning trees of $G$, although determining the structure of $\kappa(G)$ is not easy.

Dually, a configuration ${\bf c}$ is said to be \defi{superstable} if no set of (non-root) vertices can fire, so that for all $\sigma \subset \{1,2, \dots, n\}$ we have $c_i < \deg_\sigma(i)$, where
\[
    \deg_\sigma(i) = |\{j \in [n] \backslash \sigma: ij \in E(G)\}|
\]
is the number of vertices outside $\sigma$ that are adjacent to $i$.  The set of superstable configurations are also called `$G$-parking functions' in some contexts.  It turns out that there is a simple duality between superstable configurations and critical configurations.  If we let ${\bf k} = (\deg(1)-1, \deg(2)-1, \dots, \deg(n)-1)$ denote the \emph{canonical configuration} of the graph $G$, one can show that a configuration ${\bf c}$ is superstable if and only of ${\bf k} - {\bf c}$ is critical.  In particular the maximal superstables (under a natural partial order given by coordinatewise comparison) are in bijection with the minimal critical configurations of the graph $G$.

\subsection{Linear algebra}

We can rephrase many of the constructions from above in terms of some underlying linear algebra.  If $G = (V,E)$ is our graph with $V = \{0,1, \dots, n\}$ we let $\partial$ denote its (signed) incidence matrix.  Here we orient the edges of $G$ so that an edge $e = ij$ has $i \leftarrow j$ if $i < j$; hence the column of $\partial$ corresponding to $e$ has a $-1$ in the row corresponding to $i$ and a $1$ in the row corresponding to $j$ (and all other entries $0$).  Removing the row of $\partial$ corresponding to the sink vertex gives the reduced incidence matrix $\tilde \partial$.  By definition the \defi{reduced Laplacian} matrix $\tilde\Lap$  of $G$ is given by
\[
    \tilde \Lap \define \tilde \partial \tilde \partial^t.
\]

One can see that this matrix plays the role of chip-firing in the dynamical system described above.  In particular, given a configuration of chips ${\bf c}$ we have that firing a vertex $i$ corresponds to subtracting the $i$th column of $\tilde \Lap$ from ${\bf c}$ to obtain a new configuration ${\bf c}^\prime$:
\[
    {\bf c}^\prime = {\bf c} - \tilde \Lap {\bf e}_i,
\]
where ${\bf e}_i$ is the standard basis vector corresponding to the vertex $i$. Hence a configuration ${\bf c}^\prime$ is obtainable from another configuration ${\bf c}$ via a sequence of firings if ${\bf c} - {\bf c}^\prime \in \im \tilde \Lap$, in which case we write ${\bf c} \sim {\bf c}^\prime$. Furthermore there is an isomorphism of groups
\[
    \kappa(G) \cong \Z^{|V|-1} / \im \tilde \Lap.
\]

In particular we see that $\kappa(G)$ is a finite abelian group and (by the matrix tree theorem) has cardinality given by the number of spanning trees of $G$. 

\begin{example}\label{ex:running}
Consider the graph $G$ on vertex set $\{v_0,v_1,v_2,v_3\}$ depicted in Figure \ref{fig:diamond_graph}, with edge set $\{01,02,03, 12, 23\}$ (here we are suppressing set brackets and using only subscripts, so that for example the edge $\{v_1, v_2\}$ is denoted $12$). According to the orientation described above the relevant matrices are
\smallskip
\[
    \tilde{\partial} = \begin{bmatrix}
        1 & 0 & 0 & -1 &  0 \\
        0 & 1 & 0 &  1 & -1 \\
        0 & 0 & 1 &  0 &  1
    \end{bmatrix}
    \hspace{.2 in}
    \tilde\Lap = \begin{bmatrix}
         2 & -1 &  0 \\
        -1 &  3 & -1 \\
         0 & -1 &  2
    \end{bmatrix}
\]

\begin{figure}[h] 

\includegraphics[scale = 1]{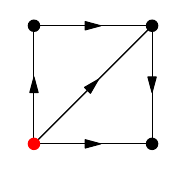}
\caption{An orientation of the Diamond graph.}
\label{fig:diamond_graph}
\end{figure}
\end{example}

\subsection{Energy minimizers and M-matrices}\label{sec:Energy}
The superstable configurations on a graph correspond to those configurations in which no set of vertices can fire simultaneously (without the number of chips at some vertex becoming negative).  More recently it was shown that superstable configurations can equivalently be seen as the unique energy minimizers among the elements in the equivalence class determined by the Laplacian (\cite{BakSho}, \cite{GuzKli}).  Here one defines a norm on chip configurations defined by the (inverse of) the reduced Laplacian.  In \cite{BakSho} Baker and Shokrieh show that the energy minimizers correspond to superstable configurations, and hence are unique per equivalence class.

In an attempt to expand the notion of chip-firing to more general contexts, Gabrielov (\cite{Gab}) studied a class of \emph{abelian avalanche models}, extending work of Dhar (\cite{Dha}).  Under this model we fix a finite set $V$ (here taken to be $[n] = \{1, \dots, n\}$) and a \emph{redistribution matrix} $\Delta$ with indices in $V$ satisfying
\[
    \Delta_{ii} > 0, \text{ for all $i$};
    \; \;
    \Delta_{ij} \leq 0, \text{ for $j \neq i$}.
\]

Inspired by considerations in physics, this captures the idea that firing a state involves losing chips at that site and increasing chips at neighboring sites. A vector ${\bf h} = (h_1, \dots, h_n) \in \Z^n$ defines a configuration, and a site $i$ is allowed to \emph{fire} if $h_i \geq \Delta_{ii}$.  In this case firing the site $i$ is defined by replacing ${\bf h}$ with the vector ${\bf h} - \Delta^T{\bf e}_i$, subtracting the $i$th row of $\Delta$ from the configuration.  Two configurations ${\bf c}$ and ${\bf d}$ are \defi{equivalent} if their difference ${\bf c} - {\bf d}$ is in $\im \Delta$, the $\Z$-image of the matrix $\Delta$.

Mimicking the setup for graphs we say that a configuration ${\bf c} = (c_1, \dots, c_n)$ is \defi{stable} if $c_i < \Delta_{ii}$ for all $i$.  If $\Delta = \tilde \Lap(G)$ is the reduced Laplacian of a graph $G$ we know any sequence of firing moves applied to any initial chip configuration ${\bf c}$ will eventually lead to a stable configuration.  It is then a natural question to ask which matrices $\Delta$ have this analogous \defi{avalanche finite} property.

Suppose $L$ is any $n \times n$ matrix.  We say that $L$ is a  \defi{Z-matrix} if $L_{ij} \leq 0$ for all $i \neq j$.   Then from \cite{GuzKli} we have the following.

\begin{prop}\cite[Definition 2.2]{GuzKli}\label{prop:Mmatrix}
Suppose $L$ is an $n \times n$ Z-matrix.  Then the following are equivalent.

\begin{enumerate}
    \item $L$ is avalanche finite.
    \item The real part of the eigenvalues of $L$ are all positive.
    \item The inverse $L^{-1}$ exists and all entries of $L^{-1}$ are non-negative.
    \item There exists a vector ${\bf x} \in {\R}^n$ with ${\bf x} \geq 0$ such that $L{\bf x}$ has all positive entries.
\end{enumerate}

\end{prop}
The equivalence of the last three conditions can be found in Plemmons \cite{Ple}, and the connection to the first condition is due to Gabrielov \cite{Gab}. If any (and hence all) of the conditions in Proposition \ref{prop:Mmatrix} hold we say that $L$ is a \defi{non-singular $M$-matrix}.  Such matrices appear in disparate fields including economics, operations research, finite element analysis \cite{Ple}.

In \cite{GuzKli} it is shown that if $L$ is an $M$-matrix then energy-minimizing configurations exists and are unique per equivalence class.  Given a configuration ${\bf c} \in \Z^n$ one defines the energy (or norm) to be
\[
    E({\bf c}) = {\bf c}^T L {\bf c}.
\]
Recall that in this set up two configurations ${\bf f},{\bf g} \in \Z^n$ are considered equivalent if ${\bf g}-{\bf f} \in \im L$, in which case we write ${\bf f} \sim {\bf g}$.  A configuration ${\bf f}$ is effective, denoted ${\bf f} \geq 0$, if we have $f_i \geq 0$ for all $i = 1, \dots, n$.  Given any ${\bf f} \in \Z^n$ with ${\bf f} \geq 0$ we consider the minimization problem
\begin{equation} \label{eq:energy}
\min_{{\bf g} \sim {\bf f}, {\bf g} \geq 0} E({\bf g}).
\end{equation}

\begin{defn}
A configuration ${\bf f} \in \Z^n$ with ${\bf f} \geq 0$ is \defi{$z$-superstable} (with respect to an $M$-matrix $L$ if for every ${\bf z} \in \Z^n$ with ${\bf z} \geq 0$ and ${\bf z} \neq {\bf 0}$ there exists $i = 1, \dots, n$ such that
\[
    {\bf f}_i - (L{\bf z})_i < 0.
\]
\end{defn} 

In other words there is no `multiset-firing' that can be performed on the given configuration ${\bf f}$ without resulting in a negative value on some site. In \cite{GuzKli}, Guzm\'{a}n and Klivans show that $z$-superstables are energy minimizers in the following sense.  

\begin{thm}\cite{GuzKli}
Let $L$ be an $n \times n$ $M$-matrix.  A vector ${\bf f} \in \Z^n$ with ${\bf f} \geq 0$ is $z$-superstable if and only if it is the (unique) minimizer of
\[
    \min_{{\bf g} \sim {\bf f}, {\bf g} \geq 0} E({\bf g}).
\]
\end{thm}

As a corollary we see that in each equivalence class defined by $L$ there exists a unique $z$-superstable configuration.  We mention that a natural notion of \emph{critical} configurations for $M$-matrices is also discussed in \cite{GuzKli}, but we will not need that here.

\section{Integral flow chip-firing and M-bases} \label{sec:dual}

Here we discuss in more detail our approach to integral flow and cycle chip-firing on a graph $G$.  As mentioned in Section~\ref{sec:intro} the basic idea is to view traditional chip-firing on $G$ as determined by the incidence matrix of $G$ and to apply matroidal/Gale duality.

We first recall some notions from lattice theory from \cite{BacDelNag}.  For this let ${\R}^m$ denote Euclidean space with the usual inner product $\langle \cdot, \cdot \rangle$. Suppose $F$ and $C$ are orthogonal subspaces of ${\R}^m$ and further suppose that both are rational (so that each has a basis consisting of vectors from ${\Q}^m$). Let $\Flow$ and $\Cut$ denote the lattices given by $\Flow = \Z^m \cap F$ and $\Cut = \Z^m \cap C$.

For a lattice $\Lambda \subset {\R}^m$ we define its \defi{dual lattice} as
\[
    \Lambda^\sharp = \{x \in {\R}^m: \langle x, \lambda \rangle \in \Z \;\text{for all $\lambda \in \Lambda$}\}. \]
From \cite{BacDelNag} we have the following.

\begin{lemma}\cite[Lemma 1]{BacDelNag}\label{lem:detiso}
With notation as above,
\[
    \Cut^\sharp/\Cut \cong \Flow^\sharp / \Flow.
\]
In fact both are isomorphic to the group $\Z^m/(\Cut \oplus \Flow)$.
\end{lemma}

The group $\Lambda^\sharp/\Lambda$ is called the \defi{determinant group} of the lattice $\Lambda$. As was pointed out in \cite{Jac}, we have another way to compute determinant groups in terms of a basis for the lattice.  In what follows recall that the \defi{Gram matrix} of a set of vectors $\{v_1, v_2, \dots, v_r\} \subset {\mathbb R}^m$ is the $r \times r$ matrix $G$ with entries given by $G_{ij} = \langle v_i, v_j\rangle$. 

\begin{lemma}\cite[Theorem 12]{Jac}\label{lem:coker}
Suppose $\Lambda$ is a rank $r$ sublattice of $\Z^m \subset {\R}^m$, and suppose $\{z_1, \dots, z_r\}$ is a $\Z$-basis for $\Lambda$ with Gram matrix $G$. Then the determinant group of $\Lambda$ is given by
\[
    \Lambda^\sharp/\Lambda \cong \Z^r/\im(G).
\]
\end{lemma}

\begin{proof}
For completeness we provide the proof from \cite{Jac}. Given the $\Z$-basis $\{z_1, \dots, z_r\}$ for $\Lambda$ let $\{w_1, \dots, w_r\} \subset {\mathbb R}^m$ denote the dual $\Z$-basis for $\Lambda^\sharp$ defined by $\langle z_i, w_j \rangle = \delta_{ij}$.  Expressing our basis in terms of the dual basis we have  
\[z_i = \sum_{j=1}^r c_{ij} w_j,\]
\noindent
where $c_{ij} = \langle z_i, z_j \rangle$. When then have
\[\Lambda^\sharp/\Lambda \cong (\Z w_1 \oplus \cdots \oplus \Z w_r)/(\Z z_1 \oplus \cdots \oplus \Z z_r)
\cong \Z^r/\im(G).\]
\end{proof}

To apply these results in our context we let $G = (V,E)$ be a finite simple graph on vertex set $V = \{0,1, \dots,n\}$ and edge set $E$. Let $\partial = \partial_G: {\R}^E \rightarrow {\R}^V $ denote the signed incidence matrix of $G$ (after specifying some orientation).  We let 
\[
    \Flow(G) = \Z^E \cap \ker(\partial)
\]
denote the \defi{lattice of integral flows}
and 
\[
    \Cut(G) = \Z^E \cap \im(\partial^T)
\]
denote the \defi{lattice of integral cuts}.  We use $\Flow$ and $\Cut$ to denote these lattices if the context is clear. Note that (assuming $G$ is connected) a $\Z$-basis for $\Cut$ is obtained by removing a single row of $\partial$ corresponding to any (say the sink) vertex.

Dually, suppose we have a $\Z$-basis $\{{\bf f}_1, \dots, {\bf f}_g\}$ for $\Flow$, where $g = |E| - |V| + 1$ is the genus of the graph.  Let $\iota$ be the $|E| \times g$ matrix with columns ${\bf f}_i$ and define $\Lap^* \define \iota^T \iota$ to be the \defi{dual Laplacian} (with respect to this choice of basis).   We recover the critical group as the cokernel of $\Lap^*$, as follows.

\begin{prop}\label{prop:dualgroup}
Suppose $G$ is any graph. Let $\iota$ be the matrix associated to any $\Z$-basis for the lattice of integral flows $\Flow(G)$, and let $\Lap^* = \iota^T\iota$ denote the associated dual Laplacian.  We then have an isomorphism of groups
\[
    \kappa(G) \cong \Z^g/\im \Lap^*.
\]
\end{prop}

\begin{proof}
If $\partial$ denotes the incidence matrix of $G$ (under some orientation) then we see that $F = \ker \partial$ and $C = \im \partial^T$ satisfy the conditions of Lemma~\ref{lem:detiso}.  Recall that we have a $\Z$-basis for $\Cut(G) = \Z^E \cap C$ given by all rows of $\partial$ corresponding to nonroot vertices.  From Lemma~\ref{lem:coker} we have that 
\[
    \Flow(G)^\sharp/\Flow(G) \cong  \Z^g/ \im \Lap^*.
\]
But also from Lemma~\ref{lem:coker} we have that
\[
    \Cut(G)^\sharp/\Cut(G) \cong \Z^n/\im \tilde \Lap \cong \kappa(G),
\]
and so the result follows from Lemma~\ref{lem:detiso}.
\end{proof}

If $G$ is a planar graph, we can take an embedding of $G$ in the plane to define a dual graph $G^*$. By definition $G^*$ has vertices given by regions in the plane determined by the embedding, and adjacency given by those regions sharing an edge.   After choosing an orientation of the graph $G$, as well as an orientation of the plane, we then have a basis for $\Flow$ given by the bounded regions of the embedding.  Proposition~\ref{prop:dualgroup} in particular tells us that $\kappa(G^*) \cong \kappa(G)$, as was first established in \cite{CorRos}.

\begin{figure}[h] 

\includegraphics[scale = 1]{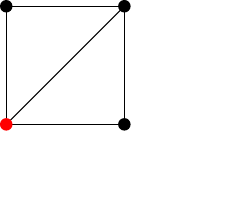}
\includegraphics[scale = 1]{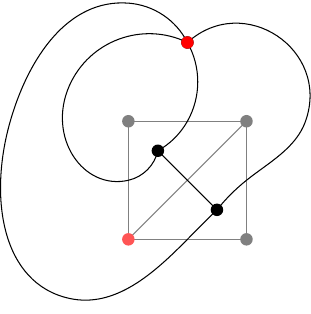}

\caption{The embedded Diamond graph and its dual.} 
 \label{fig:dual}
\end{figure}

\begin{example}
We continue with the graph $G$ from Example~\ref{ex:running}. For the choice of basis for $\Flow(G)$ described above we have the reduced Laplacian $\tilde\Lap$ given by 
\[
    \tilde\Lap = \begin{bmatrix}
         2 & -1 &  0\\
        -1 &  3 & -1\\
         0 & -1 &  2\\
    \end{bmatrix}.
\]
The critical group of $G$ is given by $\Z^3/\im \tilde\Lap \cong \Z/8\Z$, which has representatives (superstable configurations) given by $\{000, 100, 010, 001, 110, 101, 011, 020\}$.
In the dual graph depicted in Figure~\ref{fig:dual}, we take the vertex corresponding to the unbounded face as the (red) sink vertex and order the other vertices arbitrarily. In this case we have
\[
    \Lap^* = \begin{bmatrix}
         3 & -1\\
        -1 &  3
    \end{bmatrix}.
\]
The `dual' critical group is given by $\Z^2/\Lap^* \cong \Z/{8 \Z}$, and the superstable configurations of $G^*$ are $\{00, 10, 01, 20, 02, 11, 21, 12\}$.
\end{example}

\subsection{Finding an M-basis}

As we have seen the (usual, reduced) Laplacian $\tilde \Lap$ plays a prominent role in the traditional theory of chip-firing on a graph $G$.  We would like to use the dual Laplacian $\Lap^*$ to define a similar rule for `flow chip-firing' on $G$, where elements in our basis for ${\mathcal F}(G)$ lend chips to other `neighboring elements'.  However, an arbitrary choice of basis for $\Flow(G) = \ker \tilde{\partial}$ may lead to a dual Laplacian that is not a $Z$-matrix, and in particular may have positive entries off the diagonal.  We illustrate this with the following example.

\begin{example} For the graph in our running Example~\ref{ex:running} if we take the set $\left\{\begin{bmatrix} 1 \\ -1 \\ 0 \\ 1 \\ 0\end{bmatrix}, \begin{bmatrix} 0 \\ -1 \\ 1 \\ 0 \\ -1 \end{bmatrix}\right\}$ as a basis for $\ker \tilde \partial$ we obtain a dual Laplacian $\Lap^* = \begin{bmatrix} 3 & 1 \\ 1 & 3 \end{bmatrix}$.
\end{example}

To simplify notation we make the following definition.
\begin{defn}
Suppose $G$ is a connected graph and let $\Basis = \{{\bf f}_1, \dots, {\bf f}_g\}$ be a $\Z$-basis for the lattice of integral flows $\Flow(G)$. Let $\iota$ be the $|E| \times g$ matrix with columns given by the elements of $\Basis$. We say that $\Basis$ is an \defi{integral flow $M$-basis} (or simply \defi{flow $M$-basis}) if the matrix $\Lap^* = \iota^T \iota$ is an $M$-matrix.
\end{defn} 

As expected the situation is fairly straightforward for the case of planar graphs.

\begin{prop}\label{prop:planarMbasis}
Any planar graph admits a flow $M$-basis.
\end{prop}

\begin{proof}
This follows from basic properties of graph duality; we provide a proof for completeness. Suppose $G$ is a planar graph embedded in the plane and let $R_1, R_2, \dots, R_g$ denote the set of bounded regions. Orient $R_1$ arbitrarily, and for any $R_i$ that shares an edge with $R_1$ orient $R_i$ so that the edges are oriented in opposite directions.  Continue this process until all regions have been oriented.  If some $R_j$ does not share an edge with the regions oriented so far, then orient $R_j$ arbitrarily and proceed as above.  A choice of compatible orientations (in the sense that no conflicts arise in the above process) is guaranteed since ${\R}^2$ is itself orientable.  

Now for each region $R_i$ we construct a vector ${\bf f}_i \in \Z^E$ according to the orientation describe above relative to the fixed orientation on $G$ determined by the (signs of) entries of $\partial$.  Namely, the vector ${\bf f}_i$ has a $1$ in the entry corresponding to an edge $e_j$ if the orientation in $R_i$ agrees with that of $\tilde \partial$, $-1$ if the orientation is opposite, and $0$ if the edge $e_j$ does not appear in $R_i$. The collection $\{{\bf f}_1, \dots, {\bf f}_g\}$ forms a basis for $\ker \tilde \partial$ (see for instance the proof of Proposition 8 from \cite{BacDelNag}), and we let $\iota$ denote the matrix with these vectors as columns.

It follows that the matrix $\Lap^* = \iota^T \iota$ equals $\tilde \Lap(G^*)$, the reduced Laplacian of the dual graph of $G$ with respect to the given embedding.  Here the designated root of $G^*$ corresponds to the unbounded face of $G$ in this embedding. Reduced Laplacians of graphs are known to be $M$-matrices \cite{GuzKli}, so the result follows.
\end{proof}

In the case of planar graphs we can import the usual notions of superstable configurations, etc. from the theory of chip-firing on $G^*$, where $G^*$ is the dual graph defined by some embedding of $G$.   The bounded faces of $G$ (thought of as cycles) will then correspond to the nonsink vertices of $G^*$, so firing a vertex of $G^*$ corresponds to firing a cycle of $G$, and we have a reasonable notion of `flow chip-firing' (see Figure~\ref{fig:dual}).

If $G$ is not planar we still have a good notion of $\Lap^*$ (after choosing a $\Z$-basis for $\Flow(G)$) and we would like to interpret this matrix as describing some rule for chip-firing.  We first observe that the only obstruction to being an $M$-matrix is the sign pattern of the entries in $\Lap^*$.

\begin{lemma}\label{lem:MifZ}
Suppose $G$ is a graph and $\{{\bf f}_1, \dots, {\bf f}_g\}$ is a basis for $\Flow(G)$. Then the resulting dual Laplacian $\Lap^*$ is an integral flow $M$-basis if and only if the matrix $\Lap^*$ is a $Z$-matrix.
\end{lemma}

\begin{proof}
Let $\iota$ denote the matrix with columns ${\bf f}_1, \dots {\bf f}_g$.  By definition we have that $\Lap^* = \iota^T \iota$ so that $\Lap^*$ is a symmetric positive definite matrix (note that $\Lap^*$ is invertible). From this it follows that $\Lap^*$ will have positive real eigenvalues. Hence as long as $\Lap^*$ is a $Z$-matrix, it will satisfy the conditions of being an $M$-matrix listed in Proposition~\ref{prop:Mmatrix}.
\end{proof}

\begin{example}\label{ex:K5}
Consider $K_5$, the complete graph on vertex set $\{1,2,3,4,5\}$ with an orientation on edges given by $(i,j)$ whenever $i < j$.   With this orientation the matrix $\partial$ is given below (with columns ordered in lexicographical order):
\[
    \partial = \begin{bmatrix}
         1 &  1 &  1 &  1 &  0 &  0 &  0 &  0 &  0 &  0 \\
        -1 &  0 &  0 &  0 &  1 &  1 &  1 &  0 &  1 &  0 \\
         0 & -1 &  0 &  0 & -1 &  0 &  0 &  1 &  1 &  0\\
         0 &  0 & -1 &  0 &  0 & -1 &  0 & -1 &  0 &  1\\
         0 &  0 &  0 & -1 &  0 &  0 & -1 &  0 & -1 & -1 
    \end{bmatrix}
\]

A basis for $\Cut = \ker \partial$ and the resulting dual Laplacian $\Lap^* = \iota^T \iota$  are given by
\[
    \iota = \begin{bmatrix}
         1 &  0 & -1 &  0 &  1 &  0 \\
         0 & -1 &  0 &  0 & -1 &  1 \\
         0 &  0 &  0 &  1 &  0 & -1 \\
        -1 &  1 &  1 & -1 &  0 &  0 \\
         0 &  1 & -1 &  1 &  0 &  0 \\
         1 &  0 &  0 & -1 &  0 &  0 \\
         0 & -1 &  0 &  0 &  1 &  0 \\
         0 &  0 &  0 &  0 & -1 &  1 \\
         0 &  0 & -1 &  1 &  0 &  0 \\
         1 &  0 &  0 &  0 & -1 &  0
    \end{bmatrix},
\hspace{.2 in}
    \Lap^* = \begin{bmatrix}
         4 & -1 & -2 &  0 &  0 &  0 \\
        -1 &  4 &  0 &  0 &  0 & -1 \\
        -2 &  0 &  4 & -3 & -1 &  0 \\
         0 &  0 & -3 &  5 &  0 & -1 \\
         0 &  0 & -1 &  0 &  5 & -2 \\
         0 & -1 &  0 & -1 & -2 &  3
    \end{bmatrix}
\]

\begin{figure}[h] 

\includegraphics[scale = 1]{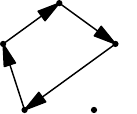} \hspace{.1 in}
\includegraphics[scale = 1]{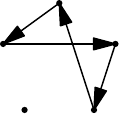} \hspace{.1 in}
\includegraphics[scale = 1]{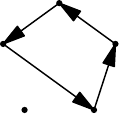} \hspace{.1 in}
\includegraphics[scale = 1]{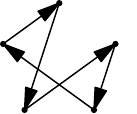} \hspace{.1 in}
\includegraphics[scale = 1]{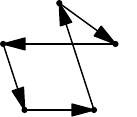} \hspace{.1 in}
\includegraphics[scale = 1]{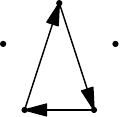}

 \caption{The cycles represented by the basis for $\ker \partial$ in Example \ref{ex:K5}.}
    
 \label{fig:cycle_basis}
\end{figure}

According to Lemma \ref{lem:MifZ} we see that $\Lap^*$ is an $M$-matrix.  Note that the sum of the entries in the third row is negative.  Hence (from results in \cite{GuzKli}) there will exist configurations that are stable under set firings but which are not $z$-superstable.

For example, one can check that the configuration ${\bf c} = (1,1,0,0,0,1)$ is $\chi$-superstable. However, firing the multiset $(5, 3, 8, 6, 4, 6)$ on ${\bf c}$ (so that site one fires $5$ times, site two fires $3$ times, etc.) results in the zero configuration.
\end{example}

To construct flow $M$-bases for general graphs we will need a slightly modified version of the Gram-Schmidt algorithm which we record below.

\begin{lemma}\label{lem:GS}
Suppose $({\bf v}_1, \dots {\bf v}_n\}$ is a linearly independent collection of vectors in ${\R}^n$.  Then there exists an orthogonal set of vectors $({\bf q}_1, \dots, {\bf q}_n\}$ with the property that for all $i =2, \dots, n$ the vector ${\bf q}_i$ is in the integer span of $\{{\bf v}_1, \dots {\bf v}_{i-1}\}$.
\end{lemma}

\begin{proof}
We apply the Gram-Schimdt algorithm but scale at each step so that the constructed vectors are in the desired integer span. In particular we set ${\bf q}_1 \define {\bf v}_1$ and for $i>1$ define the vectors ${\bf q}_i$ as
\[
    {\bf q}_i \define \left(\prod^{i-1}_{j=1}\|{\bf q}_j\|^2\right)\left({\bf v}_i-\sum^{i-1}_{j=1}\dfrac{{\bf v}_i\cdot {\bf q}_j}{\|{\bf q}_j\|^2}{\bf q}_j\right).
\]

We see that ${\bf q}_i$ is an integer linear combination of the set $\{{\bf v}_i, {\bf q}_1, \dots, {\bf q}_{i-1}\}$. By construction for all $j = 1, \dots, i-1$ each ${\bf q}_j$ is an integer linear combination of $\{{\bf v}_1, {\bf v}_2, \dots, {\bf v}_j\}$.  The vector ${\bf q}_i$ is nonzero (since the original set was linearly independent) and a calculation shows that it is orthogonal to ${\bf q}_j$ for all $j < i$.
\end{proof}

With this we prove our main algebraic result.

\begin{thm}\label{thm:latticebasis}
Suppose $\{{\bf v}_1, \dots, {\bf v}_g\}$ is a $\Z$-basis for a lattice $\Lambda \subset {\R}^d$.  Then there exists a $\Z$-basis $\{{\bf f}_1, \dots, {\bf f}_g\}$ for $\Lambda$ with the property that ${\bf f}_i \cdot {\bf f}_j \leq 0$ for all $i \neq j$.
\end{thm}

\begin{proof}
Let $\{{\bf q}_1, \dots {\bf q}_g\}$ be the set of orthogonal vectors obtained from $\{{\bf v}_1 \dots, {\bf v}_g\}$ as in the proof of Lemma~\ref{lem:GS}.  We set ${\bf f}_1 \define {\bf v}_1$ and construct the rest of our vectors inductively.  For ${\bf f}_2$ we want ${\bf f}_2 = {\bf v}_2+ t{\bf q}_1$ for some $t \in \Z$ satisfying ${\bf f}_1 \cdot ({\bf v}_2+t{\bf q}_1) \le 0$.
Since ${\bf q}_1 = {\bf f}_1$ we have that ${\bf q}_1 \cdot {\bf f}_1 > 0$ and hence we need $t \leq \dfrac{-{\bf f}_1\cdot {\bf v}_2}{{\bf f}_1 \cdot {\bf q}_1}$.  Therefore we set
\[
    t = \left\lfloor
        \dfrac{-{\bf f}_1\cdot {\bf v}_2}
              {{\bf f}_1 \cdot {\bf q}_1}
    \right\rfloor,
\]
so that ${\bf f}_2 \define {\bf v}_2+\left\lfloor\dfrac{-{\bf f}_1\cdot {\bf v}_2}{{\bf f}_1 \cdot {\bf q}_1}\right\rfloor {\bf q}_1$. Note that ${\bf f}_2$ is obtained by performing integer column operations to the matrix with columns $\{{\bf v}_1 {\bf v}_2\}$ and hence $\spn_\Z\{{\bf f}_1, {\bf f}_2\} = \spn_\Z\{{\bf v}_1, {\bf v}_2\}$. By construction we have ${\bf f}_1 \cdot {\bf f}_2 \leq 0$.

To illustrate one more step of the process, we define the vector ${\bf f}_3$ as

\[
    {\bf f}_3 \define {\bf v}_3 +
    \left\lfloor
        \dfrac{-{\bf f}_1\cdot {\bf v}_3}
              {{\bf f}_1 \cdot {\bf q}_1}
    \right\rfloor {\bf q}_1 +
    \left\lfloor
        \dfrac{-{\bf f}_2\cdot \left({\bf v}_3 +
                \left\lfloor
                    \dfrac{-{\bf f}_1\cdot {\bf v}_3}
                          {{\bf f}_1 \cdot {\bf q}_1}
                \right\rfloor
                {\bf q}_1\right)}
              {{\bf f}_2 \cdot {\bf q}_2}
    \right\rfloor
    {\bf q}_2. \]
Since ${\bf q}_2$ is orthogonal to ${\bf v}_1$ we have that 
\[{\bf f}_1 \cdot {\bf f}_3 = {\bf f}_1 \cdot {\bf v}_3 + \left\lfloor
        \dfrac{-{\bf f}_1\cdot {\bf v}_3}
              {{\bf f}_1 \cdot {\bf q}_1}
    \right\rfloor{\bf f}_1 \cdot {\bf q}_1 
    \leq 0.\]
To see that ${\bf f}_2 \cdot {\bf f}_3 \leq 0$ note that
\[
    {\bf f}_2 \cdot {\bf f}_3 = {\bf f}_2\cdot\left({\bf v}_3 +
    \left\lfloor
        \dfrac{-{\bf f}_1\cdot {\bf v}_3}
              {{\bf f}_1 \cdot {\bf q}_1}
    \right\rfloor {\bf q}_1\right) + {\bf f}_2\cdot
    \left\lfloor
        \dfrac{-{\bf f}_2\cdot \left({\bf v}_3 +
                \left\lfloor
                    \dfrac{-{\bf f}_1\cdot {\bf v}_3}
                          {{\bf f}_1 \cdot {\bf q}_1}
                \right\rfloor
                {\bf q}_1\right)}
              {{\bf f}_2 \cdot {\bf q}_2}
    \right\rfloor
    {\bf q}_2,
\]
so that

\[
    {\bf f}_2 \cdot {\bf f}_3 - 
    {\bf f}_2\cdot\left({\bf v}_3 +
    \left\lfloor
        \dfrac{-{\bf f}_1\cdot {\bf v}_3}
              {{\bf f}_1 \cdot {\bf q}_1}
    \right\rfloor {\bf q}_1\right) =  
    \left\lfloor
        \dfrac{-{\bf f}_2\cdot \left({\bf v}_3 +
                \left\lfloor
                    \dfrac{-{\bf f}_1\cdot {\bf v}_3}
                          {{\bf f}_1 \cdot {\bf q}_1}
                \right\rfloor
                {\bf q}_1\right)}
              {{\bf f}_2 \cdot {\bf q}_2}
    \right\rfloor
    {\bf f}_2\cdot{\bf q}_2.
\]

But ${\bf f}_2 \cdot {\bf q}_2 > 0$ so that

\[ 
    \left\lfloor
        \dfrac{-{\bf f}_2\cdot \left({\bf v}_3 +
                \left\lfloor
                    \dfrac{-{\bf f}_1\cdot {\bf v}_3}
                          {{\bf f}_1 \cdot {\bf q}_1}
                \right\rfloor
                {\bf q}_1\right)}
              {{\bf f}_2 \cdot {\bf q}_2}
    \right\rfloor
    {\bf f}_2\cdot{\bf q}_2 \le
    \left(
        \dfrac{-{\bf f}_2\cdot \left({\bf v}_3 +
                \left\lfloor
                    \dfrac{-{\bf f}_1\cdot {\bf v}_3}
                          {{\bf f}_1 \cdot {\bf q}_1}
                \right\rfloor
                {\bf q}_1\right)}
              {{\bf f}_2 \cdot {\bf q}_2}
    \right)
    {\bf f}_2\cdot{\bf q}_2 
\]
Hence we have
\[
    {\bf f}_2 \cdot {\bf f}_3 - 
    {\bf f}_2\cdot\left({\bf v}_3 +
    \left\lfloor
        \dfrac{-{\bf f}_1\cdot {\bf v}_3}
              {{\bf f}_1 \cdot {\bf q}_1}
    \right\rfloor {\bf q}_1\right) \le
    \left(
        \dfrac{-{\bf f}_2\cdot \left({\bf v}_3 +
                \left\lfloor
                    \dfrac{-{\bf f}_1\cdot {\bf v}_3}
                          {{\bf f}_1 \cdot {\bf q}_1}
                \right\rfloor
                {\bf q}_1\right)}
              {{\bf f}_2 \cdot {\bf q}_2}
    \right)
    {\bf f}_2\cdot{\bf q}_2 =
        {-{\bf f}_2\cdot \left({\bf v}_3 +
                \left\lfloor
                    \dfrac{-{\bf f}_1\cdot {\bf v}_3}
                          {{\bf f}_1 \cdot {\bf q}_1}
                \right\rfloor
                {\bf q}_1\right)}
\]
and we conclude
\[
    {\bf f}_2 \cdot {\bf f}_3 \le 
    -{\bf f}_2\cdot\left({\bf v}_3 +
    \left\lfloor
        \dfrac{-{\bf f}_1\cdot {\bf v}_3}
              {{\bf f}_1 \cdot {\bf q}_1}
    \right\rfloor {\bf q}_1\right) + 
    {\bf f}_2\cdot\left({\bf v}_3 +
    \left\lfloor
        \dfrac{-{\bf f}_1\cdot {\bf v}_3}
              {{\bf f}_1 \cdot {\bf q}_1}
    \right\rfloor
    {\bf q}_1\right) = 0.
\]

For the general case we let $\{{\bf q}_1, \dots {\bf q}_g\}$ be the set of vectors constructed from Lemma~\ref{lem:GS}. Assuming we have constructed vectors $\{{\bf f}_1, \dots, {\bf f}_n\}$ with the desired properties we define 
\[
    {\bf f}_{n+1} \define {\bf v}_{n+1} +
    \sum^{n}_{j=1}
        \left\lfloor
            \dfrac{-{\bf f}_j\cdot {\bf a}_{n,j}}
                  {{\bf f}_j \cdot {\bf q}_j}
            \right\rfloor
    {\bf q}_{j},
\]
where
\[
    {\bf a}_{n,j} = {\bf v}_{n+1} +
    \sum^{j-1}_{k=1}
        \left\lfloor
            \dfrac{-{\bf f}_k\cdot {\bf a}_{n,k}}
                  {{\bf f}_k \cdot {\bf q}_k}
        \right\rfloor {\bf q}_{k},
\]
with ${\bf a}_{n,1} = v_{n+1}$.

First we note that ${\bf f}_i \cdot {\bf q}_i > 0$ for all $i=1,\dots,n$, so that in particular these formulas make sense. To see this let $\alpha_j = \left\lfloor\dfrac{-{\bf f}_j\cdot {\bf a}_{i-1,j}}{{\bf f}_j \cdot {\bf q}_j}\right\rfloor$ and since the set $\{{\bf q}_1, \dots, {\bf q}_{i-1}\}$ is orthogonal observe that 
\[
    {\bf f}_i \cdot {\bf q}_i = ({\bf v}_i + \alpha_1 {\bf q}_1 + \cdots + \alpha_{i-1} {\bf q}_{i-1}) \cdot {\bf q}_i = {\bf v}_i \cdot {\bf q}_i.
\]
\noindent
  But ${\bf v}_i \cdot {\bf q}_i > 0$ since ${\bf v}_i$ and ${\bf q}_i$ are on the same side of the hyperplane $\spn\{{\bf q}_1, \dots, {\bf q}_{i-1}\}$ in the space $\spn\{{\bf q}_1, \dots, {\bf q}_i\}$. Note that ${\bf q}_i \notin \spn\{{\bf q}_1, \dots, {\bf q}_{i-1}\}$ since the set $\{{\bf q}_1, \dots, {\bf q}_i\}$ is orthogonal.  Hence for any $\beta \in {\R}$ whenever $t = \left\lfloor\dfrac{\beta}{{\bf f}_i \cdot {\bf q}_i}\right\rfloor$ we have that $t ({\bf f}_i \cdot {\bf q}_i) \leq \beta$.
 
 We claim that ${\bf f}_{k} \cdot {\bf f}_\ell \leq 0$ for all $1\leq k < \ell \leq n+1$. To see this first note that
\[{\bf f}_k \cdot {\bf f}_\ell = {\bf f}_k \cdot {\bf v}_{\ell} +
    \sum^{k}_{j=1}
        \left\lfloor
            \dfrac{-{\bf f}_j\cdot {\bf a}_{\ell-1,j}}
                  {{\bf f}_j \cdot {\bf q}_j}
            \right\rfloor
    {\bf f}_k \cdot {\bf q}_{j},
\]
since ${\bf q}_m$ is orthogonal to ${\bf f}_k$ for all $m > k$ (recall ${\bf f}_k$ is constructed as a linear combination of $\{{\bf v}_1, \dots, {\bf v}_k\}$).

But we see that 
\[\left\lfloor
            \dfrac{{-\bf f}_k\cdot {\bf a}_{\ell-1,k}}
                  {{\bf f}_k \cdot {\bf q}_k}
            \right\rfloor {\bf f}_k \cdot {\bf q}_k \le -{\bf f}_k \cdot {\bf a}_{\ell-1,k} = -{\bf f}_k \cdot \left({\bf v}_{\ell} +
    \sum^{k-1}_{j=1}
        \left\lfloor
            \dfrac{-{\bf f}_j\cdot {\bf a}_{\ell-1,j}}
                  {{\bf f}_j \cdot {\bf q}_j}
            \right\rfloor
    {{\bf q}_{j}}\right).
            \]
Hence ${\bf f}_k \cdot {\bf f}_\ell \leq 0$, as desired.

Finally it is clear that $\spn_\Z\{{\bf f}_1, \dots, {\bf f}_{n+1}\} = \spn_\Z\{{\bf v}_1, \dots {\bf v}_{n+1}\}$ since ${\bf v}_{n+1}$ was constructed by adding integer multiples of vectors from the set $\{{\bf v}_1, \dots, {\bf v}_n\}$ to the vector ${\bf v}_{n+1}$. The result follows.
\end{proof}

\begin{example}

As an small illustration of the algorithm, suppose
\[\{{\bf v}_1, {\bf v}_2, {\bf v}_3\} = \left\{\begin{bmatrix} 1 \\ 0 \\ 2 \\ 1\end{bmatrix}, \begin{bmatrix} 1 \\ 1 \\ 0 \\ 1\end{bmatrix}, \begin{bmatrix} 2 \\ 0 \\ 1 \\ 1\end{bmatrix}\right\}\]
\noindent
is a given basis for a lattice $\Lambda \in {\R}^4$.  Then Lemma \ref{lem:GS} produces the orthogonal set
\[\{{\bf q}_1, {\bf q}_2, {\bf q}_3\} = \left\{\begin{bmatrix} 1 \\ 0 \\ 2 \\ 1\end{bmatrix}, \begin{bmatrix} 4 \\ 6 \\ -4 \\ 4\end{bmatrix}, \begin{bmatrix} 396 \\ -288 \\ -144 \\ -108 \end{bmatrix}\right\}.\]

We then apply the algorithm in the proof of Theorem \ref{thm:latticebasis} to get 
\[\{{\bf f}_1, {\bf f}_2, {\bf f}_3\} = \left\{\begin{bmatrix} 1 \\ 0 \\ 2 \\ 1\end{bmatrix}, \begin{bmatrix} 0 \\ 1 \\ -2 \\ 0\end{bmatrix}, \begin{bmatrix} -3 \\ -6 \\ 3 \\ -4\end{bmatrix}\right\},\]

\noindent
a basis for $\Lambda$ with the desired properties. 
\end{example}

\begin{cor}\label{cor:Mmatrix}
Any graph $G$ admits an integral flow $M$-basis. That is, there exists a basis ${\mathcal B} = \{{\bf f}_1,  \dots, {\bf f}_g\}$ for its flow space $\Flow(G)$ such that the associated dual Laplacian $\Lap^*$ is an $M$-matrix.
\end{cor}

\begin{proof}
From Lemma~\ref{lem:MifZ} we have that $\Lap^*$ is an $M$-matrix if and only if it is a $Z$-matrix.  From Theorem~\ref{thm:latticebasis} we have that a basis $\{{\bf f}_1, \dots, {\bf f}_g\}$ for the lattice $\Flow(G)$ can always be chosen so that ${\bf f}_i \cdot {\bf f}_j \leq 0$ for all $i \neq j$.  The result follows.
\end{proof}

\subsection{\textit{z}-superstable flow configurations}

Suppose $G$ is a graph with flow $M$-basis ${\mathcal B} = \{{\bf f}_1, \dots, {\bf f}_g\}$.  Results from \cite{GuzKli} imply that the set of $z$-superstable configurations represent a unique element from each equivalence class of $\Z^g/\im({\Lap^*})$.  Recall that a configuration is $z$-superstable if no multiset of sites can be fired without resulting in a negative number of chips at some site.  In the context of dual Laplacians we will call these \defi{$z$-superstable flow configurations}, or just \defi{$z$-superstable configurations} if the context is clear.

\begin{prop}\label{prop:zsuper}
Let $G$ be a connected graph, and let ${\mathcal B} = \{{\bf f}_1, {\bf f}_2, \dots, {\bf f}_g\}$ be an integral flow $M$-basis with associated dual Laplacian $\Lap^*$.  Then the number of $z$-superstable flow configurations of $G$ is given by $|\tau(G)|$, the number of spanning trees of $G$.
\end{prop}

\begin{proof}
By results of \cite{GuzKli} discussed above, there exists a unique $z$-superstable configuration in each equivalence class determined by $\Lap^*$.  Hence the number of such $z$-superstable configurations is given by
\[
    |\Z^g/\im \Lap^*|.
\]
From Proposition~\ref{prop:dualgroup}, we have 
\[
    \kappa(G) \cong \Z^g/\im \Lap^*,
\]
and since $|\kappa(G)| = |\tau(G)|$ the result follows.
\end{proof}

Recall that for classical chip-firing on a graph $G$, the number of superstable configurations of degree $d$ is given by the number of spanning trees of $G$ with $d$ externally passive edges.  Here the \defi{degree} of a configuration ${\bf c} = (c_1, \dots, c_n)$ is simply the sum $\sum_{i=1}^n c_i$ of its entries . It is not clear if our $z$-superstable flow configurations have any combinatorial interpretation in terms of some statistic on the set of spanning trees.  See Section \ref{sec:final} for further discussion.

\begin{example}\label{ex:zsuper}
For the graph $K_5$ we use the dual Laplacian $\Lap^*$ from Example~\ref{ex:K5} which produces $125$ $z$-superstable flow configurations, corresponding to the $5^{5-2} = 125$ spanning trees of $K_5$.  This collection of multisets is clearly closed under taking subsets, so it suffices to describe the maximal elements. Here they are given by
\begin{align*}
\{000112, 000211, 010022, 010210, 010300, 020021, 020040,
020111, 021020, 021110, 030101, \\
100102, 101020, 101110, 130020, 130110, 200021, 200111, 210020, 210110, 300020, 310000\}.
\end{align*}
The degree sequence is given by $d = (1,6,19,38,39,19,3)$, where the number of $z$-superstable configurations of degree $i-1$ is given by the entry $d_i$.  
\end{example}

\section{Cycle M-bases}\label{sec:circuit}

From the above results we see any graph $G$ admits an integral flow $M$-basis, a basis for the flow space $\Flow(G)$ with the property that the associated dual Laplacian $\Lap^*$ is an $M$-matrix. In the proof of Theorem \ref{thm:latticebasis}, however, we see that the $M$-basis elements are constructed via linear combinations of a given basis and can involve many larger integers.  A natural question to ask is if one can find a flow $M$-basis for $\Flow(G)$ consisting of \emph{cycles}. More precisely we want the entries of each basis vector to be $0$ or $\pm 1$, with the the nonzero entries corresponding to some cycle in the graph $G$ (with some chosen orientation).   Note that to be an element of $\ker \tilde \partial$ we need those cycles to be oriented, so the sign of any entry will be determined by how this orientation compares to the fixed orientation on $G$.

\begin{defn}
Suppose $G$ is a connected graph.  A \defi{cycle $M$-basis} for $G$ is an integral flow $M$-basis such that each basis element corresponds to a cycle as described above.
\end{defn}

In the proof of Proposition \ref{prop:planarMbasis} we see that if $G$ is a planar graph then any embedding of $G$ gives rise to a basis for $\Flow(G)$ that consists of cycles.  Hence we in fact have the following stronger result. 
\begin{prop}\label{prop:planarcircuit}
Any planar graph admits a cycle $M$-basis.
\end{prop}

The next natural question to ask is whether \emph{any} graph $G$ admits a cycle $M$-basis.  After a version of this paper was posted to the arXiv, Chi Ho Yuen and Nathan Zelesko announced that they had shown that the graph $K_{3,5}$ does \emph{not} admit a cycle basis. Their result was established via a computer search, by considering all bases given by cycles and showing that no assignment of orientation leads to a an $M$-basis.  It would be desirable to provide a proof that does not rely on such analysis and perhaps explains `why' no such basis should exist.

On the other hand, via a computer search we have been able to find cycle $M$-bases for the two `minimal' nonplanar graphs.

\begin{prop}\label{prop:circuitK5K33}
The complete graph $K_5$ and the complete bipartite graph $K_{3,3}$ both admit cycle $M$-bases.\end{prop}

\begin{proof}
For $K_5$ we take the basis for $\Flow(K_5)$ described in Example~\ref{ex:K5}.  There we saw that the associated dual Laplacian $\Lap^*$ was an $M$-matrix.

For $K_{3,3}$ we can choose a basis for $\Flow(K_{3,3})$ given by the columns of the matrix $\iota$ below.  The corresponding dual Laplacian $\Lap^* = \iota^T \iota$ is also indicated.

\[
    \iota = \begin{bmatrix}
        -1 & 0 & 1 & 0\\
         0 & 1 &-1 & 0\\
         1 &-1 & 0 & 0\\
         1 & 0 & 0 &-1\\
         0 &-1 & 1 & 1\\
        -1 & 1 &-1 & 0\\
         0 & 0 &-1 & 1\\
         0 & 0 & 0 &-1\\
         0 & 0 & 1 & 0
    \end{bmatrix},
    \hspace{.2 in}
    \Lap^* = \begin{bmatrix}
         4 &-2 & 0 &-1\\
        -2 & 4 &-3 &-1\\
         0 &-3 & 6 & 0\\
        -1 &-1 & 0 & 4
    \end{bmatrix}
\]

Here the vertex set of $K_{3,3}$ is given by bipartition $\{0,1,2\} \cup \{3,4,5\}$, and the rows in $\iota$ are labeled lexicographically $(03,04, \dots, 24,25)$.  See Figure \ref{fig:K33} for an illustration of the elements of the corresponding cycle $M$-basis.  From Lemma \ref{lem:MifZ} we see that $\Lap^*$ is an $M$-matrix.
\end{proof}

\begin{figure}[h] 

\includegraphics[scale = 1]{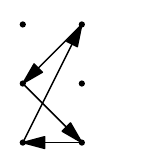}
\includegraphics[scale = 1]{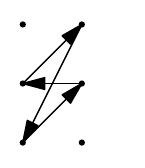}
\includegraphics[scale = 1]{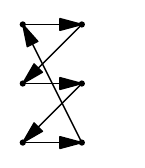}
\includegraphics[scale = 1]{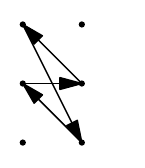}

\caption{A circuit $M$-basis for the graph $K_{3,3}$ from the proof of Proposition \ref{prop:circuitK5K33}.}
 \label{fig:K33}
   \end{figure}

\subsection{Obstructions to a cycle M-basis}

In this section we study restrictions on the structure of possible cycle $M$-bases for nonplanar graphs.  For instance, having found a cycle $M$-basis for the complete graph $K_5$ consisting of cycles (see Example \ref{ex:K5}) it is natural to ask whether it is possible to find a cycle $M$-basis consisting of all 3-cycles.  In fact this is not the case.

\begin{prop}\label{prop:triangles}
Let $n \geq 5$, and suppose $\Basis$ is a cycle $M$-basis for $K_n$. Then at least one element of $\Basis$ has more than 3 nonzero entries (so that at least one of the cycles is not a triangle).
\end{prop}

\begin{proof}
For a contradiction assume $\Basis = \{{\bf f}_1, \dots, {\bf f}_g\}$ is a cycle $M$-basis for $K_n$ consisting of all triangles. Let $\iota$ be the matrix with columns given by the elements of $\Basis$. In what follows we will also think of elements of $\Basis$ as oriented triangles.

We first claim that every row of $\iota$ must have either one nonzero entry or two nonzero entries with different signs.  To see this suppose that an edge $e$ is contained in three or more elements of $\Basis$. Then since any pair of triangles shares at most one edge, we have that some pair $\{{\bf f}_i, {\bf f}_j\}$ of these elements would have this edge oriented in the same direction.  Hence the corresponding $(i,j)$ entry in $\iota \iota^T$ would be positive, a contradiction.  

Next recall that the genus of $K_n$ is $\frac{(n-2)(n-1)}{2}$.  Hence since $\Basis$ is a cycle basis consisting only of triangles, we have that $\iota$ has a total of $\frac{3(n-2)(n-1)}{2}$ non-zero entries. However, because there is a maximum of $2$ entries per row and there are $\frac{n(n-1)}{2}$ edges, we see that there is a maximum of $n(n - 1)$ non-zero values. If $n > 6$, then $\frac{3(n-2)(n-1)}{2} > n(n - 1)$ and hence no such basis exists.

If $n = 6$, then $\frac{3(n-2)(n-1)}{2} = n(n -1) = 30$. In this case every edge must appear in exactly two triangles, with an opposite sign in each.  The sum of these cycles is the zero cycle, and hence the set $\Basis = \{{\bf f}_1, \dots, {\bf f}_g\}$ is linearly dependent. This is a contradiction to the assumption that $\Basis$ is a basis.

If $n = 5$, then we need 18 nonzero entries in 10 rows.  Hence we must have two rows in $\iota$ that have a single nonzero entry. But in this case the sum ${\bf f}_1 + \cdots + {\bf f}_g$ of the basis elements would be a vector with two nonzero entries, which cannot be an element of the flow space $\Flow(G) = \ker \tilde \partial$. The result follows.
\end{proof}

Note that since $K_4$ is planar with an embedding with bounded regions given by $3$-cycles, we see that $n \geq 5$ is tight.  A natural question to ask is how large the cycles must be in any cycle $M$-basis of a graph.  After a version of this paper was posted to the arXiv, Yuen and Zelesko \cite{YueZel} found a cycle $M$-basis for $K_{4,4}$ that does not use 8-cycles, as well as a cycle $M$-basis for $K_6$ that does not use $6$-cycles. We also note that

\begin{prop}\label{prop:fourcycles}
Let $3 \leq m \leq n$, and suppose $\Basis$ is a cycle $M$-basis for $K_{m,n}$. Then at least one element of $\Basis$ has more than $4$ nonzero entries (so that at least one of the cycles is not a $4$-cycle).
\end{prop}

\begin{proof}
Again for the sake of contradiction, assume $\Basis = \{{\bf f}_1, \dots, {\bf f}_g\}$ is a cycle $M$-basis for $K_{m,n}$ consisting of all $4$-cycles.  Let $\iota$ be the matrix with columns given by the elements of $\Basis$. In what follows we will also think of elements of $\Basis$ as oriented $4$-cycles.

As in the proof of Proposition \ref{prop:triangles} there cannot exist an edge that appears in 3 or more element of $\Basis$.  Here a pair of $4$-cycles can share two edges but these must be adjacent and furthermore must be oriented in opposite direction in the respective $4$-cycles.  Hence three or more $4$-cycle involving these edges would lead to a positive entry in the corresponding entry of $\Lap^*$.

Suppose $m=3$. For this case the matrix $\iota$ must have $8(n-1)$ nonzero entries, with at most two of these entries in each of the $3n$ rows. Hence we need $8n-8 < 6n$, so that $n < 4 $. Note that if $8n-8 = 6n$ then we get a linear dependence among the elements of $\Basis$.

For the case of $K_{3,3}$ the total number of non-zero entries in $\iota$ is $16$, and there are $9$ edges in $K_{3,3}$. We've seen that the maximum number of possible entries in any row of $\iota$ row is two (a $-1$ and a $1$), and in order for $\Basis$ to span the space $\Flow(G)$ we need $\iota$ to have a nonzero entry in every row. Hence the only way to distribute the entries would be to have 2 entries in 7 rows and only one in the last 2 rows. Similar to the $K_5$ case in \ref{prop:triangles}, this can not create a cycle basis. 

For the general case of $K_{m,n}$ with $4 \le m, n$ we need $4(mn - m - n +1) = 4mn - 4m - 4n +1$ nonzero entries in $\iota$, with at most two of these appearing in any of the $mn$ rows.  Hence we need $2mn < 4n + 4m -4$, so that $(m-2)(n-2) < 2$. Since $m, n \geq 4$ it is evident that this is never true. The result follows.
\end{proof}

\section{Further thoughts and recent developments}\label{sec:final}

In this section we discuss some open questions and further directions.  Some of these have been mentioned above but we collect them here for convenience.  

\begin{question}\label{q:circuit}
Which graphs admit a cycle $M$-basis?
\end{question}
As we have seen, any planar graph, as well as the graphs $K_{5}$ and $K_{3,3}$, admit cycle $M$-bases.  In addition, Yuen and Zelesko \cite{YueZel} have recently found cycle $M$-bases for the the Peterson graph as well as for the matroid $R_{10}$ (the smallest binary matroid that is not graphic nor cographic). In the latter case the elements of the basis correspond to \emph{circuits} of the matroid. On the other hand Yuen and Zelesko have shown via computer search that $K_{3,5}$ does not admit a cycle $M$-basis.  It would be desirable to find a human proof of this result to understand what the obstruction to a cycle $M$-basis might be.

Is there a natural class of graphs that admit a cycle $M$-bases? What about complete graphs? It is unlikely that having a cycle $M$-basis is closed under taking minors but is there some other operation that preserves this property?

Recall that if $T$ is a spanning tree of the connected graph $G$, then any edge $e$ that is not contained in $T$ gives rise to a \defi{fundamental cycle} $C_e$, consisting of $e$ along with the path in $T$ connecting the endpoints of $e$.  The collection of such $C_e$ for $e \notin T$ gives rise to a basis for $\ker \partial$ called a \defi{fundamental cycle basis} (associated to $T$).  Another natural question to determine which graphs admit a fundamental cycle basis that is also a cycle $M$-basis. Yuen and Zelesko have also addressed this question in their work, see \cite{YueZel} for futher details.

\begin{question}
Can one find an explicit bijection between the set of $z$-superstable flow configurations and the set $\tau(G)$ of spanning trees?
\end{question}

After fixing a flow $M$-basis for a graph $G$, we have seen that the set of $z$-superstable flow configurations represent a unique member of each equivalence class determined by the underlying dual Laplacian $\Lap^*$.  These configurations can be interpreted as `minimal energy' representatives, or as flow configurations for which no superset can be fired.   From Proposition \ref{prop:zsuper} we see that the set of such $z$-superstable configurations is in bijection with the set of spanning trees.  One wonders if an explicit bijection exists in the spirit of those constructed in (for instance) \cite{ChePyl}.

\begin{question}
What do the $z$-superstable flow configurations count?
\end{question}

Since the set of $z$-superstable flow configurations are in bijection with the set of spanning trees, a natural question to ask is whether the collection of such configurations is related to some statistic on the spanning trees.  

In the case of $z$-superstable flow configurations arising from the dual graph $G*$ of a planar graph $G$ the number of $z$-superstable configurations of degree $d$ are given by the number of spanning trees of $G$ with $d$ internally passive edges.  This is not the case for general graphs, and for instance the degree of a $z$-superstable configuration can be larger than $|V(G)| - 1$ (see Example \ref{ex:zsuper}).  However, it would be interesting to see if there is some notion of activity of a spanning tree, dependent on the choice of $M$-basis, that does correspond to the degree sequence of $z$-superstable flow configurations.



Finally, much of the theory and applications discussed above make sense in the more general of setting of (totally) unimodular matrices and unimodular matroids.  Here the notion of a cycle would $M$-basis would correspond to a \emph{circuit} $M$-basis, where the elements of the basis correspond to circuits (minimally dependent sets) of the underlying matroid.  It would be interesting to pursue the above constructions in those settings.


\end{document}